\documentclass[12pt]{amsart}

\author{Cristina Acciarri}

\address{via Francesco Crispi, n.81 San Benedetto del Tronto (AP) Italy}
\email{acciarricristina@yahoo.it}

\author{Aline de Souza Lima}

\address{Department of Mathematics and Statistics, Federal University of Goi\'as,
Goi\^ania-GO, 74001-970 Brazil}

\email{alinelima@mat.unb.br}

\author{Pavel Shumyatsky }

\address{ Department of Mathematics, University of Brasilia,
Brasilia-DF, 70910-900 Brazil}

\email{pavel@unb.br}

\thanks{The first author was supported by the Spanish Government, grant
MTM2008-06680-C02-02, partly with FEDER funds. She also thanks the Department of Mathematics of the University of Brasilia where this research was conducted. The second and the third authors were supported by CNPq-Brazil}

\usepackage{amssymb}
\usepackage[pdftex]{hyperref}


\title{DERIVED SUBGROUPS OF FIXED POINTS IN PROFINITE GROUPS}

\newtheorem{theorem}{\sc Theorem}[section]
\newtheorem{lemma}[theorem]{\sc Lemma}
\newtheorem{proposition}[theorem]{\sc Proposition}
\newtheorem{corollary}[theorem]{\sc Corollary}

\newcommand{\F}{\mathbb{F}}

\begin{document}

\begin{abstract} The main result of this paper is the following theorem. Let $q$ be a prime, $A$ an elementary abelian group of order $q^3$. Suppose that $A$ acts as a coprime group of automorphisms on a profinite group $G$ in such a manner that $C_G(a)'$ is periodic for each $a \in  A^{\#}$. Then $G'$ is locally finite.
\end{abstract}

\maketitle

%%%%%%%%%%%%%%%%%%%%%%%%%%%%%%%%%%%%%%%%%%%%%%%%%%%%%%%%%%%%%%%%%%%%%%%%%%%%%%%%%%%%%%%%%%%%%%%%%%%%%%%%%%%%%%%%%%%%%%

\section{Introduction}

A profinite group is a topological group that is isomorphic to an inverse limit of finite groups \cite{LM}, \cite{S2}. In the context of profinite groups all the usual concepts of group theory are interpreted topologically. In particular, by an automorphism of a profinite group we mean a continuous automorphism. A group of automorphisms $A$ of a profinite group $G$ will be called coprime if $A$ has finite order and $G$ is an inverse limit of finite groups whose orders are relatively prime to the order of $A$. Given an automorphism $a$ of a profinite group $G$, we denote by $C_G(a)$ the centralizer of $a$ in $G$, that is, the subgroup of $G$ formed by the elements fixed under $a$. This subgroup is always closed. 

The following result was proved in \cite{SA}. 

\begin{theorem}
\label{profinitoq2}
Let $q$ be a prime and $A$ an elementary abelian group of order $q^{2}$. Suppose that $A$ acts as a coprime group of automorphisms on a profinite group $G$. Assume that $C_{G}(a)$ is periodic for each $a\in A^{\#}$. Then $G$ is locally finite.
\end{theorem}  
Here and throughout the paper the symbol $A^{\#}$ stands for the set of non-identity elements of a group $A$.

Recall that using Wilson's theorem on the structure of  periodic profinite groups \cite{W1} Zelmanov  proved local finiteness of such groups \cite{Z2}. Another relevant result is Herfort's theorem that the set of prime divisors of orders of elements of a periodic profinite group is finite \cite{Herf}. It is a long-standing open problem whether every periodic profinite group has finite exponent. The proof of Theorem \ref{profinitoq2} also uses Lie-theoretic techniques in the spirit of those described in Section 3 of the present paper.

A quantitative version of Theorem \ref{profinitoq2} for finite groups was obtained earlier in \cite{KS}.

\begin{theorem}\label{q2}
Let $q$ be a prime, m a positive integer and $A$ an elementary abelian group of order $q^{2}$. Suppose that $A$ acts as a coprime group of automorphisms on a finite group $G$ and assume that $C_{G}(a)$ has exponent  dividing $m$ for each $a\in A^{\#}$. Then the exponent of $G$ is $\{m,q\}$-bounded.
\end{theorem}  

Another quantitative result of similar nature was proved in the paper of Guralnick and the third author \cite{GS}.

\begin{theorem}\label{qqq}
Let $m$ be an integer, $q$ a prime. Let $G$ be a finite $q'$-group acted on by an elementary abelian group $A$ of order $q^{3}$. Assume that $C_{G}(a)$ has derived group of exponent dividing $m$ for each $a\in A^{\#}$. Then the exponent of $G'$ is $\{m,q\}$-bounded. 
\end{theorem}
Note that the assumption that $|A|=q^{3}$ is essential here and the theorem fails if $|A|=q^{2}$, see Section 1 of \cite{GS} for more details.

In the present paper we exploit further the techniques developed in \cite{KS,SA,GS}.  If $G$ is a profinite group, we denote by $G'$ the closed subgroup generated by all commutators in $G$ (the derived group of $G$). The  main result of this paper is the following theorem.

\begin{theorem}
\label{MT}
Let $q$ be a prime, $A$ an elementary abelian group of order $q^3$. Suppose that $A$ acts as a coprime group of automorphisms on a profinite group $G$. Assume that $C_G(a)'$ is periodic for each $a\in  A^{\#}$. Then $G'$ is locally finite.
\end{theorem}
We mention that as long as it is unknown whether every periodic profinite group has finite exponent the above theorem cannot be deduced directly from Theorem \ref{qqq}. Thus, in the present paper we present an independent proof of Theorem \ref{MT}.

%%%%%%%%%%%%%%%%%%%%%%%%%%%%%%%%%%%%%%%%%%%%%%%%%%%%%%%%%%%
\section{Preliminary Results}
In this section we collect some useful results about coprime automorphisms of finite and profinite groups.
We start with two well-known lemmas (see  \cite{GO}, 6.2.2, 6.2.4).

\begin{lemma}
\label{FG1} 
Let $A$ be a  group of automorphisms of the finite group $G$ with $(|A|,|G|)=1$. 
\begin{enumerate}
\item If $N$ is any  $A$-invariant normal  subgroup of $G$, then \\$C_{G/N}(A)=C_G(A)N/N$;
\item If $H$ is any $A$-invariant $p$-subgroup of $G$, then $H$ is contained in an $A$-invariant Sylow $p$-subgroup of $G$.
\end{enumerate}
\end{lemma} 
 
\begin{lemma}
\label{FG2} 
Let $q$ be a prime, $G$ a finite $q'$-group acted on by an elementary abelian $q$-group $A$ of order $q^3$.\  Let $A_1, \dots,A_s$ be the maximal subgroups of $A$.\ If $H$ is an $A$-invariant subgroup of $G$ we have $H=\langle C_H(A_1),\dots,C_H(A_s)\rangle$.
\end{lemma}

The next results are taken from  \cite{GS}.

\begin{lemma}
\label{[N,G]}
Let $q$ be a prime, $G$ a finite $q'$-group acted on by an elementary abelian $q$-group $A$ of order $q^3$. If $N$ is any $A$-invariant normal subgroup of $G$ then $[N,G]=\left\langle [C_N(a),C_G(a)] \mid a \in A^{\#}\right\rangle$. If $[N,G]$ is nilpotent then $[N,G]= \prod[C_N(a),C_G(a)]$, where the product is taken over all $a \in A^{\#}$.
\end{lemma}

\begin{theorem}\label{T} Let $q$ be a prime and $G$ a finite $q'$-group acted on by an elementary abelian $q$-group $A$ of order $q^3$. Let $P$ be an $A$-invariant Sylow $p$-subgroup of $G'$. Then $P=\langle C_G(a)'\cap P\mid a\in A^{\#}\rangle$.
\end{theorem}

Let $\pi(G)$ denote the set of primes dividing the order of a finite group $G$. 
\begin{corollary}\label{CO}
$\pi(G')=\cup_{a\in A^{\#}}\pi(C_G(a)')$.
\end{corollary}

We will now describe a somewhat more precise version of Theorem \ref{T}. It was obtained in \cite{Sp}. Assume the hypothesis of Theorem \ref{T} and let $A_{1},\ldots,A_{s}$ be the maximal subgroups of $A$. Let $P_{1},\ldots,P_{r}$ be all the subgroups of the form 
$$
P\cap [C_{G}(A_{i}),C_{G}(A_{j})]\quad \text{for suitable}\, i,j.
$$ 
Observe that, since the intersection $A_{i}\cap A_{j}$ cannot be trivial because the rank of $A$ is 3,  any subgroup $P_{l}$, where $l=1,\ldots,r$, is contained in $[C_{G}(a),C_{G}(a)]$ for some $a\in A^{\#}$.  We have the following theorem (see Theorems 2.12 and 2.13 in \cite{Sp}).

\begin{theorem}\label{p1...pr} $P=P_{1}\cdots P_{r}$.
\end{theorem}

Using the routine inverse limit argument all the above results can be extended to the case where $G$ is a profinite group. In particular we have the following lemmas.

\begin{lemma}
\label{FG1profi} 
Let $A$ be a coprime group of automorphisms of the profinite group $G$. 
\begin{enumerate}
\item If $N$ is any  $A$-invariant normal closed subgroup of $G$, then \\$C_{G/N}(A)=C_G(A)N/N$;
\item If $H$ is any $A$-invariant pro-$p$ subgroup of $G$, then $H$ is contained in an $A$-invariant Sylow pro-$p$ subgroup of $G$.
\end{enumerate}
\end{lemma} 

\begin{lemma}
\label{FG2profi} 
Let $q$ be a prime, $G$ a profinite group coprimely acted on by an elementary abelian $q$-group $A$ of order $q^3$.\  Let $A_1, \dots,A_s$ be the maximal subgroups of $A$.\ If $H$ is an $A$-invariant subgroup of $G$ we have $H=\langle C_H(A_1),\dots,C_H(A_s)\rangle$.
\end{lemma}

Many other facts on automorphisms of finite groups admit corresponding profinite versions. In particular, later on we will use profinite versions of the above results \ref{[N,G]}, \ref{T}, \ref{CO} and \ref{p1...pr}.

%%%%%%%%%%%%%%%%%%%%%%%%%%%%%%%%%%%%%%%%%%%%%%%%%%%%%%%%%%%
\section{Useful Lie-theoretic machinery}

Let $L$ be a Lie algebra over a field ${\mathfrak k}$. Let $k$ be a positive integer and let $x_1,x_2,\dots,x_k,x,y$ be elements of $L$. We define inductively 
$$[x_1]=x_1;\ [x_1,x_2,\dots,x_k]=[[x_1,x_2,\dots,x_{k-1}],x_k].$$
 An element $a\in L$ is called ad-nilpotent if there exists a positive integer $n$ such that 
 $$[x,\underset{n}{\underbrace{a,\ldots,a}}]=0 \quad \text{for all}\,  x\in L.$$ 
 If $n$ is the least integer with the above property then we say that $a$ is ad-nilpotent of index $n$. Let $X\subseteq L$ be any subset of $L$. By a commutator in elements of $X$ we mean any element of $L$ that can be obtained as a Lie product of elements of $X$ with some system of brackets. Denote by $F$ the free Lie algebra over 
${\mathfrak k}$ on countably many free generators $x_1,x_2,\dots$. Let $f=f(x_1,x_2,\dots,x_n)$ be a non-zero element of $F$. The algebra $L$ is said to satisfy the identity $f\equiv 0$ if $f(a_1,a_2,\dots,a_n)=0$ for any $a_1,a_2,\dots,a_n\in L$. In this case we say that $L$ is PI. Now we will quote a theorem of Zelmanov \cite{Z0} which has numerous important applications to group theory (in particular see \cite{OS} for examples where the theorem is used).

\begin{theorem}\label{liealgbnilp}
Let $L$ be a Lie algebra over a field ${\mathfrak k}$ generated by $a_1,a_2,\dots,a_m$.\  Assume that $L$ is PI and that each commutator in the generators $a_1,a_2,\dots,a_m$ is {\rm ad}-nilpotent.\ Then $L$ is nilpotent.
\end{theorem}

The following theorem provides an important criterion for a Lie algebra to be PI. It was proved by Bakhturin and Zaicev for soluble groups $A$ \cite{BZ} and later extended by Linchenko to the general case \cite{LI}.

\begin{theorem}\label{PIandC} Assume that a finite group $A$ acts on a Lie algebra $L$ in such a manner that $C_L(A)$, the subalgebra formed by fixed elements, is PI. Assume further that the characteristic of the ground field of $L$ is either 0 or prime to the order of $A$. Then $L$ is PI.
\end{theorem}

The next lemma is taken from \cite{KS}. We will use it later when proving ad-nilpotency of some elements of a Lie algebra.

\begin{lemma}
\label{L14}
Suppose that $L$ is a Lie algebra, $H$ a subalgebra of $L$ gene\-ra\-ted by $r$ elements $h_1,\dots ,h_r$ such that all commutators in the $h_i$ are {\rm ad}-nilpotent in $L$. If $H$ is nilpotent, then for some number $v$ we have $[L,\underbrace{H,\dots,H}_{v}]=0$.
\end{lemma}

Now we turn to groups and  for the rest of this section $p$ will denote a fixed prime number. 
Let $G$ be any group. A series of subgroups 
\begin{equation*}
(*)\quad \quad  G=G_{1}\geq G_{2}\geq \cdots
\end{equation*}
 is called an $N_{p}$-series if $[G_{i},G_{j}]\leq G_{i+j}$ and $G_{i}^{p}\leq G_{pi}$ for all $i,j$. With any $N_{p}$-series $(*)$ of  $G$ one can associate a Lie algebra $L^{*}(G)=\oplus L^{*}_{i}$ over  the field with $p$ elements $\F_{p}$, where we view each  $L^{*}_{i}=G_{i}/G_{i+1}$ as a linear space over $\F_{p}$. If $x \in G$, let $i=i(x)$ be the largest integer such that $x \in G_i$. We denote by $x^*$ the element $xG_{i+1}$ of $L^{*}(G)$. 

\begin{lemma}[Lazard, \cite{L}]
\label{Laz} 
For any $x\in G$ we have $(ad\, x^*)^p=ad\, (x^p)^*$. Consequently, if $x$ is of finite order, then $x^*$ is {\rm ad}-nilpotent.
\end{lemma}

Let $w=w(x_1,x_2,\dots,x_n)$ be non-trivial group-word, i.e.\ a non-trivial element of the free group on free generators $x_1,x_2,\dots,x_n$. Suppose that the group $G$ has a subgroup $H$ and elements $g_1,g_2,\dots,g_n$ such that $w(g_1h_1,\dots,g_nh_n)=1$ for all $h_1,h_2,\dots,h_n\in H$. In this case we say that the law $w\equiv 1$ is satisfied on the cosets $g_1H,\dots,g_nH$ and the group $G$ satisfies a coset identity. The next proposition follows from the proof of the main theorem in the paper of Wilson and Zelmanov \cite{WZ}. 

\begin{proposition}
\label{WilZel}
Let $G$ be a group satisfying a coset identity. Then there exists a non-zero multilinear Lie polynomial $f$ over $\F_{p}$ such that  for any  $N_{p}$-series  $(*)$ of $G$ the corresponding algebra  $L^{*}(G)$ satisfies the identity $f\equiv 0$.
\end{proposition}

In general a group  $G$ has many $N_{p}$-series; one of the most important is the so called Jennings-Lazard-Zassenhaus series that can be defined as follows.   

Let $\gamma_j(G)$ denote the $j$th term of the lower central series of $G$.
Set $D_i=D_i(G)= \prod_{jp^{k}\geq i}\gamma_{j}(G)^{p^{k}}$. The subgroup $D_{i}$ is also known as  the $i$th-dimension subgroup of $G$ in characteristic $p$. These subgroups form an $N_{p}$-series of $G$ known as the  Jennings-Lazard-Zassenhaus series. Let $L_{i}=D_{i}/D_{i+1}$ and $L(G)=\oplus L_{i}$.  Then $L(G)$  is a Lie algebra over the field $\F_p$ (see \cite{GA}, Chapter 11 for more detail). The subalgebra of $L(G)$ generated by $L_{1}=D_1/D_2$ will be denoted by $L_p(G)$. The next theorem is due to Lazard \cite{L}.
\begin{lemma}
\label{PA}
Let $G$ be a finitely generated pro-$p$ group. If $L_p(G)$ is nilpotent, then $G$ is a $p$-adic analytic group.
\end{lemma}

Every subspace (or just an element) of $L(G)$ that is contained in $D_i/D_{i+1}$ for some $i$ will be called homogeneous. Given a  subgroup $H$ of  the group $G$, we denote by $L(G,H)$ the linear span in $L(G)$ of all homogeneous elements of the form $hD_{i+1}$, where $h\in D_{i}\cap H$.\ Clearly, $L(G,H)$ is always a subalgebra of $L(G)$. Moreover, it is isomorphic with the Lie algebra associated with $H$ using the $N_{p}$-series of $H$ formed by $H_{i}=D_{i}\cap H$.  
We also set $L_{p}(G,H)=L_{p}(G)\cap L(G,H)$. 

\begin{lemma}
\label{L(GH)}
Suppose that any Lie commutator in homogeneous elements $x_{1},\ldots,x_{r}$ of $L(G)$ is {\rm ad}-nilpotent.\ Let $K=\langle x_{1},\ldots,x_{r}  \rangle$ and assume that $K\leq L(G,H)$ for some subgroup  $H$ of $G$ satisfying a coset identity. Then there exists some number $v$ such that:
$$[L(G),\underset{v}{\underbrace{K,\ldots,K}}]=0.$$  

\end{lemma}
\begin{proof}
In view of Lemma \ref{L14} it is sufficient to show that $K$ is nilpotent.\ From Proposition \ref{WilZel} it follows that $K$ satisfies a multilinear polynomial identity.  Thus, by Theorem \ref{liealgbnilp} $K$ is nilpotent. 
\end{proof}

Lemma \ref{FG1profi}(1) has important implications in the context of associated Lie algebras and their automorphisms. Let $G$ be a pro-$p$ group with a coprime automorphism $a$. Obviously $a$ induces an automorphism of every quotient $D_i/D_{i+1}$. This action extends to the direct sum $\oplus D_i/D_{i+ 1}$. Thus, $a$ can be viewed as an automorphism of $L(G)$ (or of $L_p(G)$).  Set $C_i=D_i \cap C_G(a)$.\ Then Lemma \ref{FG1profi}(1) shows that 
\begin{equation}
\label{LandCand}
C_{L(G)}(a)=\oplus C_iD_{i+1}/D_{i + 1}.
\end{equation}
This implies that the properties of $C_{L(G)}(a)$ are very much related to those of $C_G(a)$ and it follows that 
\begin{equation}
\label{CLp}
L_{p}(G,C_{G}(a))=C_{L_{p}(G)}(a).
\end{equation}
In particular, Proposition \ref{WilZel} shows that if $C_G(a)$ has a certain coset identity, then $C_{L(G)}(a)$ is PI. 

\begin{lemma}
\label{periodicandL(G)PI}
Let $a$ be a coprime automorphism of a pro-$p$ group $G$. If $C_G(a)'$ is periodic, then $C_{L(G)}(a)$ satisfies a multilinear polynomial identity. 
\end{lemma}

\begin{proof}
By  Proposition \ref{WilZel} it is enough to show that $C_{G}(a)$ satisfies a coset identity. 
For each positive integer  $i$ set $$X_{i}=\{ (x,y) \mid x\in C_{G}(a), y\in C_{G}(a) \,\text{and}\, [x,y]^{p^{i}}=1 \}.$$
Each of the sets $X_{i}$ is closed in $C_{G}(a)\times C_{G}(a)$ and there are only countably many of them. We have $C_{G}(a)\times C_{G}(a)= \bigcup X_{i}$. It follows from Baire's Category Theorem (see \cite[page 200]{Kell}) that some set $X_{j}$ has a non-empty interior. In particular this set $X_{j}$ contains a set $xH_{1}\times yH_{2}$ where $H_{1}, H_{2}$ are open subgroups of $C_{G}(a)$.  Putting $H=H_{1}\cap H_{2}$  we see that the identity $[xh_{1},yh_{2}]^{p^{j}}=1$ is satisfied for all $h_{1},h_{2}\in H$  and so this is a coset identity in $C_{G}(a)$.   Applying Proposition \ref{WilZel} to $C_{G}(a)$ we conclude that $C_{L(G)}(a)$ satisfies a multilinear polynomial identity.  
\end{proof}

%%%%%%%%%%%%%%%%%%%%%%%%%%%%%%%%%%%%%%%%%%%%%%%%%%%%%%%%%%%
\section{Proof of Theorem \ref{MT}}

Now we are ready to prove Theorem \ref{MT} which we state again for the reader's convenience. 

\begin{theorem}\label{PR}
Let $q$ be a prime, $A$ an elementary abelian group of order $q^3$. Suppose that $A$ acts as a coprime group of automorphisms on a profinite group $G$. Assume that $C_G(a)'$ is periodic for each $a \in  A^{\#}$. Then $G'$ is locally finite.
\end{theorem}

\begin{proof}

Let $\pi=\pi(G')$ be the set of primes for which $G'$ has a non-trivial Sylow pro-$p$ subgroup. Note that  by the profinite version of Corollary \ref{CO} $\pi(G') = \bigcup_{a \in  A^{\#}} \pi(C_G(a)')$. According to the result of Herfort \cite{Herf} the set of primes dividing the orders of elements of a periodic profinite group is necessarily finite. Therefore each set $\pi(C_G(a)')$ is finite, since $C_G(a)'$ is periodic for all $a\in A^{\#}$. Hence $\pi$ is finite as well. Write $\pi=\left\{p_1,\ldots,p_n\right\}$.  

Choose $p\in \pi$. It follows from Lemma \ref{FG1profi}(2) that $G'$ posses an $A$-invariant Sylow pro-$p$ subgroup $P$. Let $A_{1},\ldots, A_{s}$ be the maximal subgroups of $A$ and let $P_{1},\ldots,P_{r}$ be all the subgroups of the form $P\cap[C_{G}(A_{i}),C_{G}(A_{j})]$ for suitable $i,j$. It is clear that each subgroup $P_k$ is contained in $C_{G}(a)'$ for a suitable nontrivial element $a\in A_i\cap A_j$. The nontriviality of the intersection $A_i\cap A_j$ follows from the fact that $A$ has order $q^3$ while both $A_i$ and $A_j$ have order $q^2$.
The profinite version of Theorem \ref{p1...pr} tells us that 
 \begin{equation}\label{producto} P=P_{1}P_{2}\cdots P_{r}.
 \end{equation} 
 Let $x$ be an element of $P$. In view of (\ref{producto}) we can write $x=x_{1}\cdots x_{r}$, where each $x_{i}$ belongs to some $C_{G}(a)'$ for $a\in A^{\#}$.\  Note that $r\leq q^{2}+q+1$.
 
Let $Y$ be the subgroup of $G$ generated by the orbits $x_{l}^{A}$ for $l=1,\ldots,r$. Each orbit contains at most $q^{2}$ elements so it follows that $Y$ has at most $q^{2}r$ generators.  
In order to prove that $G'$ is locally finite it is enough to show that $Y$ is finite. Indeed, once this is proved, we can say that every element $x$ in $P$ is of finite order, so that $P$ is periodic. Now let $y$ be an arbitrary element of $G'$ and let $\langle y \rangle$  be the procyclic subgroup generated by $y$. Write $\langle y\rangle=S_{1}\cdots S_{n}$, where $S_{i}$ denotes the Sylow $p_{i}$-subgroup of $\langle y\rangle$. Note that each $S_{i}$ is contained in $G'$ and from the argument above we know that every Sylow $p_{i}$-subgroup of $G'$ is periodic. So also each $S_{i}$ is periodic and it follows that  $S_{i}$ is finite, since it is actually a cyclic group. Note also that there are only finitely many of them since $\pi$ is finite. Thus, it follows that also $\langle y\rangle$ is finite and so $y$ has finite order. This holds for every element of $G'$, so $G'$ is periodic. Now Zelmanov's theorem \cite{Z2} tells us that $G'$ is locally finite, as desired.   
 
Therefore  it remains to prove that $Y=\langle x_{1}^{A},\ldots, x_{r}^{A}\rangle$ is finite. The Lie-theoretic techniques that we have described in Section 3 will now play a fundamental role. Let $L=L_{p}(Y)$ and $M=Y/\Phi(Y)$. Then $M$ is a subspace of $L$ such that $L=\langle M\rangle$. Since $Y$ is generated by at most $q^{2}r$ elements, it follows that the dimension of $M$ is at most $q^{2}r$. 

For any $a\in A^{\#}$ consider the subgroup $T_{a}=C_{G}(a)'\cap Y$ and denote by $M_{a}$ the image of $T_{a}$ in $M$. Since every element $x_i$ belongs to some $T_{a}$, it follows that $M=\sum_{a\in A^{\#}}M_a$. Note that $T_{a}$ is a pro-$p$ group that satisfies a coset identity, since $C_{G}(a)'$ is periodic. Furthermore $T_{a}$ has an $N_{p}$-series formed by $D_{i}(Y)\cap T_{a}$, so we can consider the corresponding Lie algebra which is isomorphic to  $L_{p}(Y,T_{a})$.   Note that  the algebra $\langle M_{a} \rangle$ generated by $M_{a}$ is a subalgebra  of $L_{p}(Y,T_{a})$. It follows from Proposition  \ref{WilZel} that  $\langle M_{a}\rangle$ is PI.
 
Every element of $M_{a}$ corresponds to an element of $C_{G}(a)'$ which is of finite order by the hypothesis. Moreover, for a fixed $a\in A^{\#}$, any  Lie commutator in elements of $M_{a}$ corresponds to a group commutator in elements of $C_{G}(a)'$. In view of Lemma \ref{Laz} we conclude that any Lie commutator in elements of $M_{a}$ is ad-nilpotent. We deduce from Theorem \ref{liealgbnilp} that $\langle M_{a}\rangle$ is nilpotent.  Thus, by Lemma \ref{L14}, there exists a number $v$ such that 
\begin{equation}
 \label{Maequ}
 [L, \underset{v}{\underbrace{M_{a},\ldots, M_{a}}}]=0. 
 \end{equation}   
 
Now we will extend the ground field $\F_{p}$ by a primitive $q$th root of unity $\omega$ and let  $\overline{L} = L \otimes \F_{p}[\omega]$. We view $\overline{L}$ as a Lie algebra over $\F_{p}[\omega]$ and so it is natural to identify $L$ with the $\F_{p}$-subalgebra $L\otimes 1$ of $\overline{L}$. We note that if an element $x \in L$ is ad-nilpotent,  then the ``same" element $x \otimes 1$ is also ad-nilpotent in $\overline{L}$. We will say that an element  of $\overline{L}$ is homogeneous if it belongs to $\overline{S}=S\otimes \F_{p}[\omega]$ for some homogeneous subspace $S$ of $L$.

The group $A$ acts naturally on $L$, and this action extends uniquely to $\overline{L}$. It is easy to see that $C_{\overline{L}}(a)=\overline{C_{L}(a)}$. Recall that, by (\ref{CLp}), $C_{L}(a)=L_{p}(Y,C_{Y}(a))$. The field $\F_{p}[\omega]$ contains all eigenvalues for any $a\in A$ regarded as a linear transformation of $\overline{L}$. It follows that any $A$-invariant subspace of $\overline{L}$ can be decomposed as a direct sum of $1$-dimensional $A$-invariant subspaces. Certainly the subspaces $\overline{M_{a}}$ are $A$-invariant. Using the fact that the algebra $\overline{L}$ is generated by $\overline{M}=\sum_{a\in A^{\#}}\overline{M_a}$ and  that the $\F_{p}[\omega]$-dimension of $\overline{M}$ is at most $q^{2}r$, we can choose vectors $v_{1},\ldots,v_{m}\in\cup_{a\in A^{\#}}\overline{M_a}$ with $m\leq q^{2}r$ such that each one of them is a common eigenvector for all $a\in A$ and $\overline{M}$ is spanned by these vectors.    It follows from (\ref{Maequ}) that $$[\overline{L}, \underset{v}{\underbrace{\overline{M_{a}},\ldots, \overline{M_{a}}}}]=0,$$ and, since every common eigenvector $v_{i}$ lies in some $\overline{M_{a}}$, we conclude that 

\begin{multline}
\label{eigenvectorsadnilp}
\text{ each of the vectors}\,   v_{1},\ldots,v_{m}\, \text{ is ad-nilpotent}.
\end{multline}

Now we wish to show that 
\begin{multline}
\label{commutatorhomog}
\text{if}\, l_{1},l_{2}\in \overline{L}\, \text{are common eigenvectors for all}\, a\in A\, \text{and} , \\
\shoveleft{\text{if they are
homogeneous, then}\, [l_{1},l_{2}]\, \text{is ad-nilpotent}. }
\end{multline}

Since $l_{1},l_{2}$ are common eigenvectors for  all $ a\in A$, it follows that there exist two maximal subgroups $A_{1}$ and $A_{2}$ of $A$ such that $l_{1}\in C_{\overline{L}}(A_{1})$ and $l_{2}\in C_{\overline{L}}(A_{2})$. Let $a$ be a nontrivial element in $A_{1}\cap A_{2}$. Since $C_{G}(a)'$ is periodic and $C_{L}(a)=L_{p}(Y,C_{Y}(a))$, it follows from Lemma \ref{Laz} that any homogeneous element of $[C_{L}(a),C_{L}(a)]$ is ad-nilpotent. Observe that the commutator $[l_{1},l_{2}]$ is a homogeneous element of $\overline{L}$ and so it can be written as 
$$
[l_{1},l_{2}]=y_{0}\,\otimes\,1+y_{1}\,\otimes\, \omega+\cdots +y_{q-2}\,\otimes\, \omega^{q-2},
$$ 
 for suitable homogeneous elements $y_{0},\ldots,y_{q-2}$ of $[C_{L}(a),C_{L}(a)]$. Note that the elements $y_{0},\ldots,y_{q-2}$ correspond to some elements $g_{0},\ldots,g_{q-2}$ that belong to  $C_{Y}(a)'$. Put $K=\langle y_{0},\ldots,y_{q-2}\rangle$ and $H=\langle g_{0},\ldots,g_{q-2}\rangle$. Since $H$ is a subgroup of $C_{Y}(a)$, it follows that $H$ satisfies a coset identity. Combining now the fact that $K\leq L(Y,H)$ with Lemma \ref{L(GH)} we conclude that there exists an integer $u$ such that 
$$[L,\underset{u}{\underbrace{K,\ldots,K}}]=0.$$ 
Clearly, this implies that 
$$[\overline{L},\underset{u}{\underbrace{\overline{K},\ldots,\overline{K}}}]=0.$$ 
Since the commutator $[l_{1},l_{2}]$ belongs to $\overline{K}$, we have (\ref{commutatorhomog}), as desired. 

Since $C_{Y}(a)'$ is periodic, $C_{Y}(a)$ satisfies a coset identity.  Hence, by using Lemma \ref{periodicandL(G)PI}, we obtain that $C_{L}(a)$  satisfies a certain multilinear polynomial identity.  The identity also holds in $C_{\overline{L}}(a)=\overline{C_{L}(a)}$.  Therefore Theorem \ref{PIandC} implies that $\overline{L}$  satisfies a polynomial identity. 
Combining this fact with (\ref{eigenvectorsadnilp}) and (\ref{commutatorhomog}) we apply Theorem  \ref{liealgbnilp} and deduce that $\overline{L}$ is nilpotent. Hence, $L$ is nilpotent as well.                                               

Since $Y$ is a finitely generated pro-$p$ group and  $L=L_{p}(Y)$ is nilpotent, it follows from  Lemma \ref{PA} that $Y$ is $p$-adic analytic. So $Y$ contains an open characteristic powerful subgroup $W$ (see  Chapter 3 in \cite{GA}). The subgroup $W$ is again finitely generated. Let $W_{0}$ be the set of elements of finite order in $W$. It follows that $W_{0}$ is a finite group \cite[Theorem II.4.20]{GA}. By the profinite version of Lemma \ref{[N,G]} we have 
$[W,Y]=\langle[C_{W}(a),C_{Y}(a)]\, \mid a\in A^{\#}\rangle$. Thus, it follows that $[W,Y]$ is generated by elements of finite order and so $[W,Y]\leqslant W_{0}$. This implies that the quotient group $Y/W_{0}$ is central-by-finite since $W$ is central in $Y$ modulo $W_{0}$ and $W $ has finite index in $Y$. By Schur's theorem \cite[p. 102]{rob} the image of $Y'$ in $Y/W_{0}$ is finite and we conclude that $Y'$ is also finite. Note that $Y/Y'$ is abelian and generated  by finitely many elements of finite order, since every $x_{i}^{A}$ is contained in $C_{G}(a)'$ for a suitable $a\in A^{\#}$. It follows that the quotient $Y/Y'$ is finite and so is $Y$. This concludes the proof of the theorem.   
\end{proof}

%%%%%%%%%%%%%%%%%%%%%%%%%%%%%%%%%%%%%%%%%%%%%%%%%%%%%%%%%%%%%%%%%%%%%%%%%%%%%%%%%%%%%%%%%%%%%%%%%%%%%%%%%%%%%%%%%%%%%%

\end{document}